\documentclass[12pt]{article}
\usepackage{amssymb,amsfonts,amsthm}
\usepackage{amstext}
\usepackage{epsfig}
\usepackage{amsmath}
\usepackage{indentfirst}
\usepackage[a4paper,margin=2.3cm]{geometry}

\usepackage{color}
\usepackage[latin1]{inputenc}

\def\openC{{\rm C\kern-.18cm\vrule width.8pt height 7pt depth-.2pt \kern.18cm}}
\def\openN{{{\rm I}\kern-.16em {\rm N}}}
\def\openR{{{\rm I}\kern-.16em {\rm R}}}
\def\openT{{{\rm T}\kern-.42em {\rm T}}}
\def\openZ{{{\rm Z}\kern-.28em{\rm Z}}}

\newtheorem{thm}{Theorem}[section]

\theoremstyle{definition}
\newtheorem{defn}[thm]{Definition}
\newtheorem{rem}[thm]{Remark}
\newtheorem{ex}[thm]{Example}

\begin{document}

\title{\textbf{Oscillation and nonoscillation criteria for impulsive delay differential equations with Perron integrable righthand sides} \vspace{-2pt}
\author{\sc
M. Ap. Silva; M. Federson \,\,\&\,\, M. C. Gadotti }}
\date{}
\maketitle \vspace{-30pt}
\bigskip
\begin{center}
\parbox{13 cm}{{\small We present new criteria for the existence of oscillatory and nonoscillatory solutions of measure delay differential equations with impulses. We deal with the integral forms of the differential equations using the Perron and the Perron-Stieltjes integrals. Thus the functions involved can have many discontinuities and be of unbounded variation and yet we obtain good results which encompass those in the literature. Examples are given to illustrate the main results.}}
\end{center}
delay differential equations; impulses; Perron integral; Perron-Stieltjes integral. 

\thispagestyle{empty}

\section{Introduction}

Measure differential equations have been investigated by many authors, as for instance, W. Schmaedeke \cite{Schm}, P. Das and R. Sharma \cite{Das-71, Das-72} and others.  The main purpose of the  concept of measure differential equations is the description of systems exhibiting discontinuous solutions caused by the impulsive behavior of the differential system. The theory of measure differential equations allows us to encompass, for example, differential equations with impulses (\cite[Theorem 3.1]{FMS2}) as well as dynamic equations on time scales (\cite[Theorem 4.3]{FMS1}).

On the other hand, oscillations are an important property of particles in Quantum Mechanics and other areas of Physics, with applications to many applied sciences such as Engineering, Finance, etc.  In the present  paper, we deal with a class of delayed measure differential equations subject to impulse action and we present oscillation and nonoscillation criteria of solutions.

Consider the impulsive problem given by
\begin{equation} \label{eq3}
\left\{
\begin{array}{l}
Dy= - p(t)y(t - \tau)Dg \medskip\\
y(t_k^+)-y(t_k)=b_ky(t_k), \quad k\in \mathbb{N},
\end{array}
\right.
\end{equation}
where $ \tau >0$ is a constant, $\mathbb{N}={\{1,2, \ldots }\}$, $p:[t_0, \infty) \to \mathbb{R}$, $g:[t_0, \infty) \to \mathbb{R}$ is a regulated function which is left-continuous and continuous at the points of impulses $t_k$, $k \in \mathbb{N}$, $Dy$ and $Dg$ stand for the distributional derivatives of the functions $y$ and $g$ in the sense of distributions of L. Schwartz and, moreover, $ t_0 < t_1 < \ldots < t_k < \ldots $ are fixed points and $\lim_{k \to \infty}t_k= \infty$, $b_k \in (- \infty, -1) \cup (-1, \infty)$ are constants, for $k \in \mathbb{N}$, and
for each compact subset $[a,b]$ of $[t_0, \infty)$, the Perron-Stieltjes integral $\int_{a}^{b} p(s) \, dg(s)$ exists. 

Our goal here is to provide conditions for the oscillation of all solutions of \eqref{eq3} and for the existence of nonoscillatory solutions of \eqref{eq3}. 

Because the main feature of Perron integrable functions is to cope with many discontinuities and highly oscillating functions (that is, functions of unbounded variation, as e.g. $F:[0,1] \to \mathbb{R}$, $F(t)= t^2 \sin(\frac{1}{t}), 0<t\leq1$ and $F(t)=0,\, t=0$), our results encompass those from the classical theory for impulsive differential equations, see e.g. \cite{Berezansky, Berezansky2} and references therein.

Oscillation criteria for 
nonimpulsive differential equations of the type
$$\dot{x}(t)= - p(t) x(t - \tau),$$
where $p:[t_0, \infty) \to \mathbb{R}$ is a 
continuous function and $\tau$ is a positive 
constant can be found in \cite{Agarwal,GL}. 
In \cite{Yan}, the authors deal with impulsive equation 
\begin{equation}\label{eqyan}
 \dot{y}(t) + \sum_{i=1}^{n} p_i(t) y(t - \tau_i(t))=0, \quad t \neq t_k
 \end{equation} 
 and $y(t_k^+) - y(t_k) = b_k y(t_k), \quad k=1, 2, \ldots$, where $0 \leq t_0 < t_1 < \ldots < t_k < \ldots$ are fixed points with $\lim_{k \to \infty} t_k = \infty$, $p_i \in C([t_0, \infty), \mathbb{R})$ are locally summable functions and $\tau_i \in ([t_0, \infty), [0, \infty))$ are Lebesgue measurable functions and $t - \tau_i (t) \to \infty$ as $t \to \infty, i =1, 2, \ldots, n$, and $b_k \in (-\infty, -1) \cup (-1, \infty)$ are constants for $k=1, 2, \ldots$.
Assuming these hypotheses, they prove that all solutions of  \eqref{eqyan} are oscillatory if and only iff all solutions of a nonimpulsive equation related to \eqref{eqyan} are oscillatory. However, the authors do not prove any result on the oscillation and the existence of  nonoscillatory solutions.


In the present paper we deal with the integral forms of the differential equations using the Perron and Perron-Stieltjes integrals, so that highly oscillatory righthand sides having many discontinuities can be taken into account. We also present two examples in order to illustrate the main results.

\section{The generalized Perron integral}

In this section, we only mention some basic results  on the generalized Perron integral. More details can be found in \cite{Monteiro,Schwabik}.

 Let $[a,b]$ be an interval with $-\infty < a < b < \infty$. A pair $(\tau, J)$ of a point $\tau \in \mathbb{R}$ and a compact interval $J \subset \mathbb{R}$ is called a tagged interval and $\tau$ is called the tag of $J$. 
A finite collection $D={\{(\tau_j,J_j); j= 1, 2, \cdots, k}\}$ of tagged intervals is called a system in $[a,b]$, if $\tau_j \in J_j \subset [a,b]$ for every $j=1, \cdots, k$, and the intervals $J_j$ are nonoverlapping.
A system $D={\{(\tau_j, J_j); j=1, 2, \cdots , k \}}$ is called a partition of $[a,b]$, if  $\bigcup_{j=1}^{k} J_j= [a,b]$.

Given a function $\delta: [a,b] \to (0, \infty)$, called a gauge on $[a,b]$, a tagged interval $(\tau, J)$ with $\tau \in [a,b]$ is said to be $\delta$-fine, whenever $J \subset (\tau - \delta(\tau), \tau + \delta(\tau))$.
A system (in particular, a partition) $D={\{(\tau_j,J_j); j= 1, 2, \cdots, k}\}$ is said to be $\delta$-fine if each point-interval pair $(\tau_j, J_j)$ is $\delta$-fine for $j=1, 2, \cdots, k$. 

Let $|\cdot|$ denote any norm in $\mathbb{R}$.
\begin{defn}\label{d}
 
A function $f: [a,b] \to \mathbb{R}$ is called Perron integrable or Kurzweil integrable with respect to function $g : [a,b] \to \mathbb{R}$, if there is $I \in \mathbb{R}$ such that given $\epsilon > 0$, there is a gauge $\delta$ on $[a,b]$ such that

\[
\left| \displaystyle\sum_{j=1}^k f(\tau_j) \left[g(\alpha_j) - g(\alpha_{j-1}) \right]  - I \right|  < \epsilon,
\]

 for every $\delta$-fine partition $D={\{\tau_j, [\alpha_{j-1},\alpha_j]; j=1, \cdots, k}\}$ of $[a,b]$.
\end{defn}

The real number $I \in \mathbb{R}$ in Definition \ref{d} is called the Perron-Stieltjes integral of $f$ with respect to $g$ over $[a,b]$ and we denote it by $\int_a^b f(t) \, dg(t)$, which is well-known to generalize the Riemann-Stieltjes and Lebesgue-Stieltjes integrals. Clearly, in the particular case where $g(t)= t$, for $t \in \mathbb{R}$, then the number $I$ is the Perron integral.  We denote by $\mathcal{K} ([a,b],\mathbb{R})$ the space of functions of $[a,b] \times [a,b]$ in $\mathbb{R}$ that are integrable Kurzweil. 

\begin{thm}
Given $g:[a, b] \to \mathbb{R}$ a regulated function, for $f_1 \in \mathcal{K} ([a,b],\mathbb{R})$ and $f_2:[a,b] \to \mathbb{R}$ with $f_1 = f_2$ almost everywhere, we have $f_2 \in \mathcal{K} ([a,b],\mathbb{R})$ and 
$$\displaystyle\int_a^b f_2(t)\, dg(t) = \displaystyle\int_a^b f_1(t) \,dg(t).$$
\end{thm}

The proof of the above result follows the same steps as the proof presented in \cite{Chaim} for $g(t)=t$. 

For $f \in \mathcal{K} ([a,b],\mathbb{R})$, we define $\widetilde{f}(t)= \int_a^t f(s)ds$, $t \in [a,b]$. In $\mathcal{K} ([a,b],\mathbb{R})$ we consider the equivalence relation $\widetilde{f_1}=\widetilde{f_2}$ and we denote by $K([a,b], \mathbb{R})$ the space of the equivalence classes. We have $\widetilde{f_1}=\widetilde{f_2}$ iff $f_1=f_2$ almost everywhere, as in \cite{Chaim}. 

Given  $f \in K([a,b], \mathbb{R})$, we define the Alexiewicz norm (see \cite{norma})
$$||f||_A= \displaystyle\sup_{a \leq t \leq b} \left\| \displaystyle\int_a^t f(s) ds \right\| = ||\widetilde{f}||.$$

 Now, we present a dominated convergence theorem for the Perron integral. For a proof of it, the reader may want to look at a more general statement in \cite[Corollary 1.31]{Schwabik}.
 
\begin{thm}\label{THM}
Let ${\{f_n}\}_{n\in \mathbb N}$ be a sequence of Perron integrable functions from $[a,b]$ to $\mathbb R$ and $f:[a,b]\to\mathbb R$ be a function such that  the following conditions hold
\begin{itemize}
\item [i)] $f_n \to f$ almost everywhere (in the sense of the Lebesgue measure, in which case we write simply a.e.);
\item [ii)] there exists a nonegative Perron integrable  function $g:[a,b]\to\mathbb R$ such that $|f_n| \leq g$ a.e., for all  $n\in\mathbb N$.
Then $f$ is Perron integrable and
$$\displaystyle\lim_{n \to \infty}\int_a^b f_n(t)dt =\int_a^b f(t)dt.$$

\end{itemize}
\end{thm} 


\section{Preliminaries}

We deal with the following definitions. At first, recall the concept of a regulated function. Any function $f:[a,\infty) \to \mathbb{R}$ is called regulated, whenever the lateral limits 
$f(t^-)=\lim_{s\to t^-}f(s),$ $ t\in(a,\infty),$ and $ f(t^+)=\lim_{s\to t^+}f(s),$ $ t\in[a,\infty)$,
exist. We write $f\in G([a,\infty),\mathbb{R})$ in this case. If moreover $f:[a, \infty) \to \mathbb{R}$ is left-continuous, then we write $f \in G^-([a,\infty), \mathbb{R})$. If $f \in G^-([a,\infty), \mathbb{R})$ is continuous at the points of impulses $t_k$, then we write $f \in \widetilde{G}^-([a,\infty), \mathbb{R})$.

It is a know fact that any function $f:[a, \infty) \to \mathbb{R}$ of locally bounded variation is also regulated. Moreover, any function $f$ in $G([a,b], \mathbb{R})$ can be uniformly approximated by step functions. See \cite{Honig}, for instance, for a proof of these facts and other properties of regulated functions. 

\begin{defn} \label{d1}
A function $y \in G^-([t_0 - \tau,\infty), \mathbb{R})$ is a Perron-Stieltjes solution of \eqref{eq3} or simply a solution of \eqref{eq3}  on $[t_0 , \infty)$ if  for every $s_1, s_2 \in [t_0, \infty)$,
$$y(s_2) - y(s_1) = - \displaystyle\int_{s_1}^{s_2} p(t) y(t - \tau) dg(t) + \displaystyle\sum_{t_0 < t_k \leq s_2} I_k (y(t_k))$$
and, for every $k \in \mathbb{N}$, $I_k (y(t_k))=  y(t_k^+) - y(t_k^ -) = y(t_k^+) - y(t_k) = b_k y(t_k),$
that is, we suppose $y$ is left-continuous at $t=t_k$ and the lateral limits $y(t_k ^+)$ exist, for every $k \in \mathbb{N}$.
\end{defn}

\begin{rem} Definition \ref{d1} is equivalent to saying that 
$y(t)$ satisfies $$Dy(t)= - p(t) y(t - \tau) Dg(t),$$
for almost all  $t \in [t_0, \infty),$ and $y(t_k^+)= y(t_k) + I_k (y(t_k)), \, k \in \mathbb{N}.$ This fact follows from the Fundamental Theorem of Calculus for the Perron-Stieltjes integral, see \cite{FTC} for the Fundamental Theorem of Calculus and \cite{Schwabik}, page 91, for the notion of a Perron-Stieltjes solution.
\end{rem}

Since $y$ and $g$ are regulated functions, they are also Darboux \cite[Theorem 3.6]{Honig} integrable and consequently Lebesgue integrable. Therefore,  $y$ and $g$ define distributions in $\mathbb{R}$ (see \cite{hormander}). 

A function $y$  is said to be {\em eventually positive almost everywhere}, if there exists $ T \geq t_0$ such that $y(t)>0$, for almost all $t \geq T$. Similarly, a function $y$ is said to be {\em eventually negative almost everywhere}, if there exists $ T \geq t_0$ such that $y(t)<0$, for almost all $t \geq T.$ 

\begin{defn}
Given $g:[t_0, \infty) \to \mathbb{R}$ a regulated function which is left-continuous and continuous at the points of impulses $t_k$, $k \in \mathbb{N}$, we say that solution of equation \eqref{eq3} is {\em nonoscillatory}, if it is either eventually positive almost everywhere or eventually negative almost everywhere. A non trivial solution of equation \eqref{eq3} which is not nonoscillatory is called {\em oscillatory}. 
\end{defn}
Note that according to our definition of oscillation, the function $y(t)= 1 + \sin(t)$ is nonoscillatory.

\section{An oscillation criterion}

In this section, we present an oscillation criterion for a class of measure delay differential equations with impulses.

Together with  the impulsive differential equation \eqref{eq3}, we also consider  the following auxiliary measure delay differential equation
\begin{equation}  \label{eq4}
Dx=-P(t)x(t-\tau)Dg
\end{equation}
where 
\begin{equation}\label{P}
P(t)= \prod_{t- \tau \leq t_k< t}(1+b_k)p(t), \quad  t \geq t_0.
\end{equation}
We say that solution $x$ of equation \eqref{eq4} is {\em nonoscillatory}, if it is either eventually positive almost everywhere or eventually negative almost everywhere.

Note that, if $b_k < -1$, for $k \in \mathbb{N}$, then any solution $y$ of \eqref{eq3} is oscillatory. Indeed, since
$y(t_k^+)\,y(t_k)= (1 +b_k) y^2(t_k) < 0,$
it follows that any nontrivial solution $y$ of \eqref{eq3} is neither eventually positive nor eventually negative  almost everywhere.

The following result is very important because it provides us with a transformation. We can transform a measure differential equation with impulsives in a nonimpulsive differential equation, preserving the oscillation or nonoscillation of its solutions.

\begin{thm} \label{teo1}
Assume that $g \in \widetilde{G}^-([t_0, \infty), \mathbb{R})$ and the Perron-Stieltjes integral \linebreak $\int_{a}^{b} p(s)dg(s)$ exists for each subinterval $[a,b]$ of $[t_0, \infty)$. 
\begin{enumerate}
\item[(i)] If  $x$ is a  solution of \eqref{eq4} on $[\sigma, \infty)$, $\sigma \geq t_0$, then 
\begin{equation}\label{yy}
y(t)= \displaystyle\prod_{\sigma \leq t_k < t}(1+b_k)^{-1}x(t)
\end{equation}
 is a solution of \eqref{eq3} on $[\sigma, \infty)$.	
\item[(ii)] If $y$ is a solution of \eqref{eq3} on $[\sigma, \infty)$, $\sigma \geq t_0$, then 
\begin{equation}\label{xx}
x(t)= \displaystyle\prod_{\sigma \leq t_k <t}(1+b_k)y(t) 
\end{equation}
 is a solution of \eqref{eq4} on $[\sigma, \infty)$.
\end{enumerate}
 In particular, $x$ is nonoscillatory solution of \eqref{eq4} if and only if $y$ is a nonoscillatory solution of \eqref{eq3}. 
\end{thm}

\begin{proof}
Consider $T> 0$. Let $x$ be a  solution of \eqref{eq4} and $y$ be defined by \eqref{yy}. Then for each $t \geq T+\tau$,\begin{eqnarray*}
Dy + p(t)y(t-\tau)Dg & = & \prod_{T \leq t_k <t}(1 +b_k)^{-1}Dx + p(t) \prod_{T \leq t_k <t-\tau}(1+b_k)^{-1}x(t-\tau)Dg\\
& = & \prod_{T \leq t_k < t}(1+b_k)^{-1} \left(Dx +  \prod_{t- \tau \leq t_k <t}(1+b_k)p(t)x(t-\tau)Dg \right)\\
& = & \prod_{T \leq t_k < t}(1+b_k)^{-1} \left(Dx +  P(t)x(t-\tau)Dg \right)= 0.
\end{eqnarray*}
Moreover, for every $k \in \mathbb{N}$, we have
\[
\begin{array}{rcl}
y(t_k^+)&=&\displaystyle\lim_{t \to t_k^+}\prod_{T \leq t_j <t}(1+b_j)^{-1}x(t)=\prod _{T \leq t_j \leq t_k}(1+b_j)^{-1}x(t_k)\\
y(t_k)&= & \prod_{T\leq t_j< t_k}(1+b_k)^{-1}x(t_k).
\end{array}
\]
Thus, for every $t_k \geq T+\tau$, $ k \in \mathbb{N}$, we have $y(t_k^+)=(1+b_k)y(t_k).$ Therefore $y$ is solution of \eqref{eq3}.

Now we prove that the transformation \eqref{yy} keeps the oscillatory behavior of solutions. Without loss of generality, suppose that $x$ is an eventually positive almost everywhere solution of \eqref{eq4}, i.e. $x(t)>0 $ for almost all $t \geq T \geq t_0$ and, therefore, $x(t)>0$ for almost all $t \geq T+ \tau$. Hence $x(t - \tau)>0$ for almost all $t \geq T+ \tau$ and from  \eqref{yy} we have
$y(t_k)=\prod_{T\leq t_j< t_k}(1+b_k)^{-1}x(t_k)>0$ and $y(t_k^+)=(1+b_k)y(t_k)>0$. Thus $y$ is eventually positive almost everywhere solution of \eqref{eq3}.

Conversely, if  $y$ is a solution of \eqref{eq3}, then we show that $x$ defined by \eqref{xx} satisfies equation  \eqref{eq4}. 
Indeed, using \eqref{P}  for each $t \geq T + \tau$ 
\begin{eqnarray*}
Dx + P(t) x(t-\tau)Dg & = & \prod_{T \leq t_k <t} (1+b_k)Dy  + P(t)  \prod_{T \leq t_k< t-\tau}(1+b_k)y(t-\tau)Dg\\
& = & \prod_{T \leq t_k<t}(1+b_k) \left[ Dy+p(t)y(t-\tau)Dg \right]=0.
\end{eqnarray*}
Let $y$ be a nonoscillatory solution of \eqref{eq3}. Without loss of generality, we suppose that $y(t)$ is eventually positive almost everywhere. Then there exists $T \geq t_0$ such that $y(t)>0$ for almost all $t \geq T$. In particular, $y(t)>0$ for almost all $t \geq T + \tau$, i.e. $y(t - \tau) >0$ for almost all $t \geq T+ \tau$. Since $y$ is a nonoscillatory solution of \eqref{eq3}, we get $b_k>-1$ for every large $k \in \mathbb{N}$. 
Then  the function $x$ defined by \eqref{xx} for $t \geq T+\tau$ is a solution of \eqref{eq4} satisfying $x(t)>0$ for almost all $t \geq T+\tau$, that is, $x$ is an eventually positive almost everywhere solution. 
\end{proof}

			


The next result gives sufficient conditions for all solutions of \eqref{eq3} to be oscillatory.

\begin{thm} \label{Teorema*}
Assume  that $g \in \widetilde{G}^-([0, \infty), \mathbb{R})$ and $g$ is nondecreasing function. Suppose that $p \in K([a,b], \mathbb{R})$ and the Perron-Stieltjes integral $\int_{a}^b p(s) \, dg(s)$ exists for each subinterval $[a,b]$ of $[0, \infty)$. If
\begin{equation}\label{C1}
\displaystyle\limsup_{t \to \infty} \displaystyle\int_{t - \tau}^t \displaystyle\prod_{s - \tau \leq t_k < s} (1+b_k) p(s) \, dg(s) >1,
\end{equation}
 where $p(t)\, Dg(t) >0$  for almost every $t \in [0, \infty)$, then all solutions of \eqref{eq3} are oscillatory.
\end{thm}

\begin{proof}
	Suppose that $y$ is a nonoscillatory solution of \eqref{eq3}. Without loss of generality, we suppose that it is eventually positive  almost everywhere. Then there exists $T \geq 0$ such that $y(t)>0$ for almost all $t \geq T$ and $b_k>-1$ for every large $k$. Thus, the function $P$ defined by \eqref{P}
	has the same sign as $p$ for $t\geq T$. 
	
From Theorem \ref{teo1}, equation \eqref{eq4} also has a positive  almost everywhere solution 
$x$ on $[T, \infty)$. Integrating equation \eqref{eq4} from $t - \tau$ to $t$, we obtain
\begin{equation}\label{int}
x(t) - x(t - \tau) = - \displaystyle\int_{t - \tau}^t P(s) x(s - \tau) dg(s).
\end{equation}
Note that, since $x(t) >0$ for almost all $t \geq T$, $x(t-\tau)> 0$ for almost all $t \geq T+\tau$.
Thus, for almost all $ t \geq T + \tau$, we have
$$x(t) - x(t - \tau) = - \displaystyle\int_{t - \tau}^t P(s) x(s - \tau) dg(s) \leq 0,$$
because $P$ and $Dg$ have the same sign.
Therefore, $x(t) \leq x(t - \tau)$ for almost all  $ t \geq T + \tau$.
 From here and \eqref{int} 
we get
\[
x(t-\tau)\displaystyle\int_{t - \tau}^t P(s) \, dg(s)\leq \displaystyle\int_{t - \tau}^t P(s) x(s - \tau)\, dg(s)
\]
for almost all $t \geq T+ \tau$.
Thus, 
\[
0 \geq x(t) + x(t- \tau) \left[ \displaystyle\int_{ t - \tau}^t P(s) \, dg(s) - 1 \right], \quad  t \geq T+\tau,
\]
which is a contradiction with \eqref{C1}.
Therefore $y$ is an oscillatory solution of \eqref{eq3}.
\end{proof}

\begin{rem}
Note that when $p$ is of bounded variation, the Perron-Stieltjes integral $\int_{a}^b p(s) \, dg(s)$ exists. (See \cite[Theorem 4]{substi}).

\end{rem}

\section{A nonoscillation criterion}

In this section, we study the existence of nonoscillatory solutions
for measure delay differential equations with impulses of type \eqref{eq3} with $g(s)=s$. We assume that $p$ is positive almost everywhere in $[t_0, \infty)$. Therefore, equation \eqref{eq3} can be rewritten as
\begin{equation}\label{equacao2}
\left\{
\begin{array}{l}
\dot{y}(t)=-p(t)y(t - \tau) \medskip\\
y(t_k^+)-y(t_k)=b_ky(t_k), \quad k\in \mathbb{N},
\end{array}
\right.
\end{equation}
satisfying the same hypothesis of equation \eqref{eq3}, that is, 
$ t_0 < t_1 < \ldots < t_k < \ldots $ are fixed points and $\displaystyle\lim_{k \to \infty}t_k= \infty$, for $k \in \mathbb{N}$, $b_k \in (- \infty, -1) \cup (-1, \infty)$ are constants and $\tau >0$ is a constant, and for each subinterval $[a,b]$ of $[t_0, \infty)$, the Perron integral $\int_{a}^{b} p(s) \, ds$ exists.

Let $u$ be a Perron integrable function over $[t_0,\infty)$ 
and consider the integral equation
\begin{equation} \label{ua}
u(t)= P(t)\exp\left(\displaystyle\int_{t-\tau}^{t}u(s)ds \right).
\end{equation}
We say that solution of equation \eqref{ua} is {\em nonoscillatory}, if it is either eventually positive almost everywhere or eventually negative almost everywhere.

\begin{thm} \label{lem}
Assume that the Perron integral $\int_{a}^{b} p(s) \, ds$ exists for each subinterval $[a,b]$ of $[t_0, \infty)$ and $b_k>-1$, for every $k \in \mathbb{N}$. If $ p $ is positive almost everywhere, then the following statements are equivalent:
\begin{itemize}
\item[\rm{(i)}] Equation \eqref{equacao2} has a nonoscillatory solution. 
\item[\rm{(ii)}] Equation \eqref{ua} has a nonoscillatory solution.
\item[\rm{(iii)}] The sequence $(u_k)$ of Perron integrable functions 
\[
\begin{array}{rcl}
u_1(t)&=&P(t)=\displaystyle\prod_{t-\tau \leq t_k< t}(1+b_k)p(t)\\
u_{k+1}(t)&=&P(t) \exp \left(\displaystyle\int_{t-\tau}^tu_k(s)ds \right),
\end{array}
\]
defined for almost every $t \geq T \geq t_0$ and $k \in \mathbb{N}$, converges pointwisely  almost everywhere in  $[T, \infty)$. 
\end{itemize}
\end{thm}

\begin{proof}
Let us prove the implication $(i) \Rightarrow (ii)$. Let $y$ be a nonoscillatory solution of \eqref{equacao2}. By Theorem \ref{teo1}, $x$ defined by \eqref{xx} is an nonoscillatory solution of \eqref{eq4}. Without loss of generality, we suppose that $x(t)>0$ for almost all $t\geq T \geq t_0$. 
Set 
\[u(t)= -\frac{\dot{x}(t)}{x(t)}
\] for almost all $t \geq T.$ 
Then $u$ is defined almost everywhere in $[T, \infty)$ and it is not difficult to prove that 
$u$ satisfies \eqref{ua} almost everywhere in $[T, \infty) $. Thus $u$ is eventually  positive almost everywhere solution of	\eqref{ua}  in $[T, \infty)$.
		
		
Now, we prove the implication $(ii) \Rightarrow (i)$. Let $u$ be a nonnegative solution of \eqref{ua}  almost everywhere in $[T, \infty)$. Set
\[
x(t)= - \exp \left(-\int_T^t u(s) ds \right),
\] 
for $t \geq T.$
Then $x$ is a solution of \eqref{eq4}  which is negative almost everywhere in $[T, \infty)$ and, hence, $y(t)=\prod_{T \leq t_k <t}(1+b_k)^{-1}x(t)$ is a solution of \eqref{equacao2} which is negative almost everywhere in $[T, \infty)$. 
			
Let us prove $(ii) \Rightarrow (iii)$. Let $u(t)$ be a solution of \eqref{ua} which is nonnegative almost everywhere in $[T, \infty)$. Then, for almost all $ t \geq T$,
\begin{align*}
u_1(t)&= P(t) \leq u(t)\\
u_1(t)&\leq P(t) \exp \left(\displaystyle\int_{t- \tau}^tu_1(s)ds \right)=u_2(t)\leq P(t) \exp \left(\displaystyle\int_{t- \tau}^tu(s)ds \right)= u(t).
\end{align*}
By induction, one can prove that, for almost all $t \geq T$ and all $k \in \mathbb{N}$, 
\begin{equation}\label{bound}
0< u_k(t) \leq u_{k+1}(t) \leq u(t).
\end{equation}
Thus the sequence ${\{u_k(t)\}}$ converges pointwisely almost everywhere to, say, a function $\widetilde{u}(t)$, that is,
$\lim_{k \to \infty}u_k(t)=\widetilde{u}(t) \leq u(t), $
which means that $(iii)$ holds. 
				
Finally, we prove the implication $(iii) \Rightarrow (ii)$. 
Consider the sequence $(u_k)$ and let $\displaystyle\lim_{k \to \infty} u_k(t)= u(t)$ for almost all $t \geq T$. 
Then  \eqref{bound} holds
and the functions $u_k$, $k \in \mathbb{N}$, are uniformly bounded by a positive function $u$ on $[t- \tau, t]$ for almost all $t \geq T$. Therefore, by Dominated Convergence Theorem for the Perron integral, see Theorem \ref{THM} 
we obtain that $u$ satisfies \eqref{ua} and is eventually positive almost everywhere. The proof is complete.
\end{proof}
		
The next result concerns a nonoscillation criteria for \eqref{equacao2}.

\begin{thm}\label{teo2}
	Assume that  the Perron integral $\int_{a}^{b} p(s) \, ds$ exists for every subinterval $[a,b]$ of $[t_0, \infty)$ and $b_k>-1$, for $k \in \mathbb{N}$. Suppose $ p $ is positive almost everywhere and there is $T \geq t_0$ such that  
	\begin{equation}\label{x}
	\displaystyle\int_{t-\tau}^t\displaystyle\prod_{s-\tau \leq t_k<s}(1+b_k)p(s)ds \leq  \displaystyle\frac{1}{e}, \,  \, \quad t \geq T.
	\end{equation}
	Then equation \eqref{equacao2} admits a nonoscillatory solution. 
\end{thm}
\begin{proof}
 Consider the sequence  $(u_k)$ of Perron integrable functions defined in Theorem \ref{lem}-(iii). 
By \eqref{x}, we obtain
\begin{equation*}
\displaystyle\int_{t - \tau}^tu_1(s)ds=\displaystyle\int_{t-\tau}^tP(s)ds \leq \displaystyle\frac{1}{e}, \qquad t \geq T.
\end{equation*}
Moreover, for almost all $t \geq T$, we have 
\begin{eqnarray*}
u_2(t)&=& P(t)\exp \left(\displaystyle\int_{t - \tau}^tu_1(s)ds \right) \leq P(t)e^{\frac{1}{e}} \leq P(t)\,e.\\
u_3(t)&=&P(t)\exp \left(\displaystyle\int_{t - \tau}^tu_2(s)ds \right) \leq P(t) \exp \left(e \displaystyle\int_{t - \tau}^tP(s)ds \right) \leq P(t) \, e.
\end{eqnarray*}
By induction, for almost all $t \geq T$ and all $k \in \mathbb{N}$, we obtain
\[
u_k(t) \leq u_{k+1}(t) = P(t) \exp \left(\displaystyle\int_{t - \tau}^tu_k(s) ds \right) \leq  P(t)\exp \left(\displaystyle\int_{t -\tau}^t P(s) \, e \, ds \right)\leq P(t)\, e.
\]
Thus, $0\leq u_k(t) \leq u_{k+1}(t) \leq  P(t)\, e,$ for almost all $ t \geq T$ and all $k \in \mathbb{N},$
which implies that the sequence $(u_k)$ is convergent almost everywhere in $[T, \infty)$ to the nonoscillatory solution of \eqref{ua}. By Theorem  \ref{lem}, equation
\eqref{equacao2} has a nonoscillatory solution.
\end{proof}

\section{Examples}

In order to illustrate the main results, we present some examples.
\begin{ex}
		Consider the impulsive delay differential equation 
		\begin{equation} \label{eqex}
		\left\{
		\begin{array}{l}
		Dy(t)= - t^4 \chi_{[4, \infty)\setminus \mathbb{Q}}(t) y(t - 2)Dg(t), \quad t \in [4, \infty) \setminus {\{t_1, t_2, \ldots, t_k, \ldots}\} \medskip\\
		y(t_k^+)-y(t_k)=\displaystyle\frac{1}{2}y(t_k), \quad k\in \mathbb{N},
		\end{array}
		\right.
		\end{equation}
		where $\chi_A$ denotes the characteristic function of a set $A \subset \mathbb{R}$, $p(t)= t^4 \chi_{[4, \infty)\setminus \mathbb{Q}}(t)$, $g(t)=2t^3$ and $b_k= \displaystyle\frac{1}{2}$, for all $k \in \mathbb{N}$.
		Suppose that there is at most one point of impulse effect in each interval $[t - \tau, t)$ where $\tau >0$ is given. Thus,
		\begin{equation*}
		\displaystyle\int_{t -2}^t \displaystyle\prod_{s - 2 \leq t_k <s} (1+b_k) p(s) \, dg(s)= (1+b_k) \displaystyle\int_{t - 2}^t p(s) \, \, dg(s), \qquad \mbox{if} \, \, t_k \in [t - 2, t) 
		\end{equation*}
		or
		\begin{equation*}
		\displaystyle\int_{t -2}^t \displaystyle\prod_{s - 2 \leq t_k <s} (1+b_k) p(s) \, dg(s) = \displaystyle\int_{t - 2}^t p(s) \, dg(s), \qquad  \mbox{if} \, \, t_k \notin [t - 2, t).
		\end{equation*}
	
The nonimpulsive differential equation related to \eqref{eqex} is
\begin{equation}\label{eqex1}
Dx(t)= - P(t) x(t-2) Dg(t),
\end{equation}
where $P(t)=\prod_{t-2 \leq t_k < t} (1+b_k)p(t) =\frac{3}{2} t^4 \chi_{[4, \infty)\setminus \mathbb{Q}}(t)$.		
By Theorem \ref{teo1}, equation \eqref{eqex} is oscillatory if and only if equation \eqref{eqex1} is oscillatory. Note that $p(t)= t^4 \chi_{[4, \infty)\setminus \mathbb{Q}}(t)$ and  $ Dg(t)=\dot{g}(t)= 6t^2$, since $g$ is a continuous function. Thus $p$ and $\dot{g}$ have the same signal. Moreover, $g(t)= 2 t^3$ is clearly regulated, since it is continuous.
				
It is evident that the Lebesgue integral $\int_a^b p(s) dg(s) = \int_a^b p(s) \dot{g}(s)$ exists for every $[a,b] \subset [4, \infty)$. Hence, the Perron integral $\int_a^b p(s)dg(s)$ exists on each compact  subinterval of $[4, \infty)$. Furthermore,
\[
\displaystyle\limsup_{t \to \infty}\int_{t-2}^t \displaystyle\prod_{s-2 \leq t_k <s} (1 + b_k)  p(s) \, dg(s) =\displaystyle\limsup_{t \to \infty}\int_{t-2}^t 9 s^6 ds >1.
\]
Therefore, by Theorem \ref{Teorema*}, all solutions of \eqref{eqex} are oscillatory, and by 
Theorem \ref{teo1} the same holds for \eqref{eqex1}.
\end{ex}

\begin{ex}
Consider the equation
\begin{equation}\label{equa}
\dot{x}(t)= -\displaystyle\frac{1}{t^2} \displaystyle\chi_{ [3, \infty) \setminus \mathbb{Q}}(t) x(t-1), \qquad t \in [3, \infty),
\end{equation}
where $p(t)=\displaystyle\frac{1}{t^2} \displaystyle\chi_{ [3, \infty) \setminus \mathbb{Q}}(t)$.
Since $p$ is Lebesgue integrable over $[3, \infty)$ with finite integral, the Perron integral $\int_3^{\infty} p(s)ds$ exists.  Moreover, 
$$\displaystyle\int_{t-1}^{t} p(s)ds =  \displaystyle\frac{1}{t(t-1)} < \displaystyle\frac{1}{e}, \, \, \, \, \,   \, \, \, \,  t \geq 3.$$
Thus equation \eqref{equa} satisfies the hyphoteses of Theorem \ref{teo2} and, hence, it has a nonoscillatory solution. 
	
	
\end{ex}

\section*{Acknowledgements} 

The first author was supported by FAPESP (grant 2015/12489-0) and CAPES (grant PROEX 9422567-D); the second author was supported by CNPq (grant 309344/2017-4) and FAPESP (grant 2017/13795-2) and the third author by FAPESP (grant 2018/15183-7).


\noindent 	M. Ap. Silva  and M. Federson \\
Departamento de Matem\'atica, ICMC\\
Universidade de São Paulo - S\~ao Carlos, Caixa Postal 668,\\
13560-970 S\~ao Carlos SP, Brazil\\ Emails: marielle@usp.br; federson@icmc.usp.br 

\vspace{2mm}

\noindent M. C. Gadotti \\
Departamento  de  Matemática, IGCE  \\ 
 Universidade  Estadual Paulista, \\
Avenida 24A 1515, 13506-700 Rio Claro SP, Brazil \\
Email: mc.gadotti@unesp.br

\end{document}